\documentclass{elsart}%
\usepackage{amsfonts}
\usepackage{amsmath}
\usepackage{amssymb}
\usepackage{graphicx}%
\providecommand{\U}[1]{\protect\rule{.1in}{.1in}}
\newtheorem{theorem}{Theorem}

\newenvironment{proof}[1][Proof]{\noindent\textbf{#1.} }{\ \rule{0.5em}{0.5em}}
\begin{document}
%
%TCIMACRO{\TeXButton{Begin frontmatter}{\begin{frontmatter}}}%
%BeginExpansion
\begin{frontmatter}%
%EndExpansion
%

%TCIMACRO{\TeXButton{Title}{\title
%{Difference equations of q-Appell polynomials}}}%
%BeginExpansion
\title{Difference equations of q-Appell polynomials}%
%EndExpansion
%

%TCIMACRO{\TeXButton{Author}{\author{Nazim I. Mahmudov}}}%
%BeginExpansion
\author{Nazim I. Mahmudov}%
%EndExpansion
%

%TCIMACRO{\TeXButton{Address}{\address{Eastern Mediterranean University,
%Department of Mathematics
%Gazimagusa, TRNC, Mersin 10, Turkey \\
%Email: nazim.mahmudov@emu.edu.tr}}}%
%BeginExpansion
\address{Eastern Mediterranean University,
Department of Mathematics
Gazimagusa, TRNC, Mersin 10, Turkey \\
Email: nazim.mahmudov@emu.edu.tr}%
%EndExpansion
%

%TCIMACRO{\TeXButton{Make title}{\maketitle}}%
%BeginExpansion
\maketitle
%EndExpansion
%

%TCIMACRO{\TeXButton{Begin abstract}{\begin{abstract}} }%
%BeginExpansion
\begin{abstract}
%EndExpansion
In this paper, we study some properties of the $q$-Appell polynomials,
including the recurrence relations and the $q$-difference equations which
extend some known calssical ($q=1$) results. We also provide the recurrence
relations and the $q$-difference equations for $q$-Bernoulli polynomials,
$q$-Euler polynomials, $q$-Genocchi polynomials and for newly defined
$q$-Hermite polynomials$,$ as special cases of $q$-Appell polynomials.%
%TCIMACRO{\TeXButton{End abstract}{\end{abstract}}}%
%BeginExpansion
\end{abstract}%
%EndExpansion
%

%TCIMACRO{\TeXButton{Begin keyword(s)}{\begin{keyword}}}%
%BeginExpansion
%EndExpansion
%

%TCIMACRO{\TeXButton{End frontmatter}{\end{frontmatter}}}%
%BeginExpansion
\end{frontmatter}%
%EndExpansion

\section{Introduction}

He and Ricci \cite{He} obtained the differential equations of the Appell
polynomials via the factorization method. Moreover, they found differential
equations satisfied by Bernoulli and Euler polynomials as a special case.
Afterward, Da-Qian Lu found differential equations for generalized Bernoulli
polynomials in \cite{lu}. Recently, several interesting properties and
relationships involving the classical Appell type polynomials were
investigated \cite{ricci2}-\cite{simsek}.

The proof given by He and Ricci used the factorization method, which based on
raising and lowering operators techniques. Note that the raising operators are
not available for general polynomials, although lowering operators always
exist. The proof of the main results given here for $q$-Appell polynomials
does not use raising operators.

In this paper, we derive $q$-difference equations for $q$-Appell polynomials
$A_{n,q}\left(  x\right)  $ defined in Al-Salam \cite{sa1}. As special cases
of $q$-Appell polynomials, we also provide the $q$-difference equations for
$q$-Bernoulli polynomials $B_{n,q}\left(  x\right)  $, $q$-Euler polynomials
$E_{n,q}\left(  x\right)  $, $q$-Genocchi polynomials and for newly defined
$q$-Hermite polynomials $H_{n,q}\left(  x\right)  $.

We briefly recall some of the properties of these polynomials. The Appell
polynomials can be defined by considering the following generating function:%
\begin{equation}
A\left(  x,t\right)  :=A_{q}\left(  t\right)  e_{q}\left(  tx\right)
=\sum_{n=0}^{\infty}A_{n,q}\left(  x\right)  \dfrac{t^{n}}{\left[  n\right]
_{q}!},\ \ 0<q<1, \label{ap1}%
\end{equation}
where%
\[
A_{q}\left(  t\right)  :=\sum_{n=0}^{\infty}A_{n,q}\dfrac{t^{n}}{\left[
n\right]  _{q}!},\ \ \ A\left(  0\right)  \neq0,\ \ \
\]
is analytic function at $t=0$ , and $A_{n,q}:=A_{n,q}\left(  0\right)  ,$ and
$e_{q}\left(  t\right)  =\sum_{n=0}^{\infty}\dfrac{t^{n}}{\left[  n\right]
_{q}!}.$

Differentiating generation equation (\ref{ap1}) with respect to $x$ and
equating coefficients of $t^{n},$ we obtain%
\[
D_{q,x}A_{n,q}\left(  x\right)  =\left[  n\right]  _{q}A_{n-1,q}\left(
x\right)  .
\]
Then the lowering operator $\Phi_{n}=\dfrac{1}{\left[  n\right]  _{q}}D_{q,x}$
satisfies the following operational relation:%
\[
\Phi_{n}A_{n,q}\left(  x\right)  =A_{n-1,q}\left(  x\right)  .
\]
It follows that%
\begin{equation}
A_{n-k,q}\left(  x\right)  =\left(  \Phi_{n-k}\circ...\circ\Phi_{n}\right)
A_{n,q}\left(  x\right)  =\frac{\left[  n-k\right]  _{q}!}{\left[  n\right]
_{q}!}D_{q,x}^{k}A_{n,q}\left(  x\right)  . \label{app6}%
\end{equation}

\section{ Recursion formulas and $q$-difference equations}

In this section, we derive a difference equation for the $q$-Appell
polynomials $A_{n,q}\left(  x\right)  $ and give the recurrence relations and
difference equations for the $q$-Appell polynomials.

\begin{theorem}
\label{Thm:a1}The following linear homogeneous recurrence relation for the
$q$-Appell polynomials holds true:%
\begin{align*}
A_{n,q}\left(  qx\right)   &  =\frac{1}{\left[  n\right]  _{q}}\sum_{k=0}%
^{n}\left[
\begin{array}
[c]{c}%
n\\
k
\end{array}
\right]  _{q}\alpha_{n-k}q^{k}A_{k,q}\left(  x\right)  +xq^{n}A_{n-1,q}\left(
x\right) \\
&  =\frac{1}{\left[  n\right]  _{q}}\alpha_{0}q^{n}A_{n,q}\left(  x\right)
+q^{n}\left(  x+\alpha_{1}q^{-1}\right)  A_{n-1,q}\left(  x\right)  +\frac
{1}{\left[  n\right]  _{q}}\sum_{k=0}^{n-2}\left[
\begin{array}
[c]{c}%
n\\
k
\end{array}
\right]  _{q}\alpha_{n-k}q^{k}A_{k,q}\left(  x\right)  .
\end{align*}

\end{theorem}

\begin{proof}
See formula (\ref{app5}) in the proof of Theorem \ref{Thm:a2}.
\end{proof}

\begin{theorem}
\label{Thm:a2}Assume that%
\begin{equation}
t\dfrac{D_{q,t}A_{q}\left(  t\right)  }{A_{q}\left(  qt\right)  }=\sum
_{n=0}^{\infty}\alpha_{n}\dfrac{t^{n}}{\left[  n\right]  _{q}!}. \label{ass1}%
\end{equation}
The $q$-Appell polynomials $A_{n,q}\left(  x\right)  $ satisfy the
$q$-difference equation%
\begin{gather*}
\dfrac{\alpha_{n}}{\left[  n\right]  _{q}!}D_{q,x}^{n}A_{n,q}\left(  x\right)
+\dfrac{q\alpha_{n-1}}{\left[  n-1\right]  _{q}!}D_{q,x}^{n-1}A_{n,q}\left(
x\right)  +...+\dfrac{q^{n-2}\alpha_{2}}{\left[  2\right]  _{q}!}D_{q,x}%
^{2}A_{n,q}\left(  x\right) \\
+\dfrac{q^{n-1}\alpha_{1}}{\left[  1\right]  _{q}!}D_{q,x}A_{n,q}\left(
x\right)  +\dfrac{q^{n}\alpha_{0}}{\left[  0\right]  _{q}!}A_{n,q}\left(
x\right)  +xq^{n}D_{q,x}A_{n,q}\left(  x\right)  -\left[  n\right]
_{q}A_{n,q}\left(  qx\right)  =0.
\end{gather*}

\end{theorem}

\begin{proof}
Differentiating generating equation%
\begin{equation}
A_{q}\left(  qx,t\right)  =A_{q}\left(  t\right)  e_{q}\left(  tqx\right)
=\sum_{n=0}^{\infty}A_{n,q}\left(  qx\right)  \dfrac{t^{n}}{\left[  n\right]
_{q}!} \label{ap4}%
\end{equation}
with respect to $t$ and multiplying the obtained equality by $t,$ we get the
following two equations%
\begin{align*}
tD_{q,t}A_{q}\left(  qx,t\right)   &  =D_{q,t}\left(  A_{q}\left(  t\right)
e_{q}\left(  tqx\right)  \right)  =\left(  D_{q,t}A_{q}\left(  t\right)
\right)  e_{q}\left(  tqx\right)  +qxA_{q}\left(  qt\right)  e_{q}\left(
tqx\right) \\
&  =A_{q}\left(  x,qt\right)  \left[  t\dfrac{D_{q,t}A_{q}\left(  t\right)
}{A_{q}\left(  qt\right)  }+tqx\right] \\
tD_{q,t}A_{q}\left(  qx,t\right)   &  =t\sum_{n=0}^{\infty}\left[  n\right]
_{q}A_{n,q}\left(  qx\right)  \dfrac{t^{n-1}}{\left[  n\right]  _{q}!}%
=\sum_{n=0}^{\infty}\left[  n\right]  _{q}A_{n,q}\left(  qx\right)
\dfrac{t^{n}}{\left[  n\right]  _{q}!}.
\end{align*}
Now from the assumption (\ref{ass1}) it follows that%
\begin{align}
\sum_{n=0}^{\infty}\left[  n\right]  _{q}A_{n,q}\left(  qx\right)
\dfrac{t^{n}}{\left[  n\right]  _{q}!}  &  =A_{q}\left(  x,qt\right)  \left[
t\dfrac{D_{q,t}A_{q}\left(  t\right)  }{A_{q}\left(  qt\right)  }+tqx\right]
\nonumber\\
&  =\sum_{n=0}^{\infty}q^{n}A_{n,q}\left(  x\right)  \dfrac{t^{n}}{\left[
n\right]  _{q}!}\left[  \sum_{n=0}^{\infty}\alpha_{n}\dfrac{t^{n}}{\left[
n\right]  _{q}!}+tqx\right] \nonumber\\
&  =\sum_{n=0}^{\infty}\sum_{k=0}^{n}\left[
\begin{array}
[c]{c}%
n\\
k
\end{array}
\right]  _{q}\alpha_{k}q^{n-k}A_{n-k,q}\left(  x\right)  \dfrac{t^{n}}{\left[
n\right]  _{q}!}+x\sum_{n=0}^{\infty}q^{n+1}A_{n,q}\left(  x\right)
\dfrac{t^{n+1}}{\left[  n\right]  _{q}!}. \label{ap2}%
\end{align}
Equating coefficients of $t^{n}$ in equation (\ref{ap2}), we obtain%
\begin{equation}
\left[  n\right]  _{q}A_{n,q}\left(  qx\right)  =\sum_{k=0}^{n}\left[
\begin{array}
[c]{c}%
n\\
k
\end{array}
\right]  _{q}\alpha_{k}q^{n-k}A_{n-k,q}\left(  x\right)  +x\left[  n\right]
_{q}q^{n}A_{n-1,q}\left(  x\right)  . \label{app5}%
\end{equation}
Inserting (\ref{app6}) into (\ref{app5}) we get%
\begin{align*}
\left[  n\right]  _{q}A_{n,q}\left(  qx\right)   &  =\sum_{k=0}^{n}\left[
\begin{array}
[c]{c}%
n\\
k
\end{array}
\right]  _{q}\alpha_{k}q^{n-k}\frac{\left[  n-k\right]  _{q}!}{\left[
n\right]  _{q}!}D_{q,x}^{k}A_{n,q}\left(  x\right)  +x\left[  n\right]
_{q}q^{n}\frac{\left[  n-1\right]  _{q}!}{\left[  n\right]  _{q}!}%
D_{q,x}A_{n,q}\left(  x\right) \\
&  =\sum_{k=0}^{n}\frac{q^{n-k}}{\left[  k\right]  _{q}!}\alpha_{k}D_{q,x}%
^{k}A_{n,q}\left(  x\right)  +xq^{n}D_{q,x}A_{n,q}\left(  x\right) \\
&  =\left(  \sum_{k=0}^{n}\frac{q^{n-k}}{\left[  k\right]  _{q}!}\alpha
_{k}D_{q,x}^{k}+xq^{n}D_{q,x}\right)  A_{n,q}\left(  x\right)  .
\end{align*}

\end{proof}

\section{$q$-Bernoulli polynomials}

The Bernoulli polynomials $B_{n,q}\left(  x\right)  $ are defined (see
\cite{sa2}, \cite{mah1}) starting from the generating function:%
\[
B_{q}\left(  x,t\right)  :=\frac{t}{e_{q}\left(  t\right)  -1}e_{q}\left(
tx\right)  =\sum_{n=0}^{\infty}B_{n,q}\left(  x\right)  \dfrac{t^{n}}{\left[
n\right]  _{q}!},\ \ \ \left\vert t\right\vert <2\pi,
\]
and consequently, the Bernoulli numbers $b_{n,q}:=B_{n,q}\left(  0\right)  $
can be obtained by the generating function:%
\[
B_{q}\left(  t\right)  :=\frac{t}{e_{q}\left(  t\right)  -1}=\sum
_{n=0}^{\infty}b_{n,q}\dfrac{t^{n}}{\left[  n\right]  _{q}!}.
\]

\begin{theorem}
\label{Thm:b1}The following linear homogeneous recurrence relation for the
$q$-Bernoulli polynomials holds true:%
\[
B_{n,q}\left(  qx\right)  =q^{n}\left(  x-\frac{1}{q\left[  2\right]  _{q}%
}\right)  B_{n-1,q}\left(  x\right)  -\frac{1}{\left[  n\right]  _{q}}%
\sum_{k=0}^{n-2}\left[
\begin{array}
[c]{c}%
n\\
k
\end{array}
\right]  _{q}q^{k-1}b_{n-k,q}B_{k,q}\left(  x\right)
\]

\end{theorem}

\begin{theorem}
\label{Thm:b2}The $q$-Bernoulli polynomials $B_{k,q}\left(  x\right)  $
satisfy the $q$-difference equation%
\begin{gather*}
\dfrac{b_{n,q}}{q\left[  n\right]  _{q}!}D_{q,x}^{n}B_{n,q}\left(  x\right)
+\dfrac{b_{n-1,q}}{\left[  n-1\right]  _{q}!}D_{q,x}^{n-1}B_{n,q}\left(
x\right)  +...+q^{n-3}\dfrac{b_{2,q}}{\left[  2\right]  _{q}!}D_{q,x}%
^{2}B_{n,q}\left(  x\right) \\
-q^{n}\left(  x-\frac{1}{q\left[  2\right]  _{q}}\right)  D_{q,x}%
B_{n,q}\left(  x\right)  +\left[  n\right]  _{q}B_{n,q}\left(  qx\right)  =0.
\end{gather*}

\end{theorem}

\section{$q$-Euler polynomials}

The Euler numbers $e_{n,q}$ can be defined by the generating function%
\[
E_{q}\left(  t\right)  :=\frac{te_{q}\left(  t\right)  }{e_{q}\left(
2t\right)  -1}=\sum_{n=0}^{\infty}e_{n,q}\dfrac{t^{n}}{\left[  n\right]
_{q}!}.
\]
The Euler polynomials $E_{n,q}\left(  x\right)  $ (see \cite{mah1}) can be
defined by the generating function%
\[
E_{q}\left(  x,t\right)  :=\frac{2}{e_{q}\left(  t\right)  +1}e_{q}\left(
tx\right)  =\sum_{n=0}^{\infty}E_{n,q}\left(  x\right)  \dfrac{t^{n}}{\left[
n\right]  _{q}!},\ \ \ \left\vert t\right\vert <\pi.
\]
The connection to the Euler numbers is given by%
\[
e_{n,q}=2^{n}E_{n,q}\left(  \frac{1}{2}\right)  .
\]

\begin{theorem}
\label{Thm:e1}The following linear homogeneous recurrence relation for the
$q$-Euler polynomials holds true:%
\[
E_{n,q}\left(  qx\right)  =\frac{1}{2}\sum_{k=0}^{n-1}\left[
\begin{array}
[c]{c}%
n-1\\
k
\end{array}
\right]  _{q}q^{k}E_{n-k-1}E_{k,q}\left(  x\right)  +xq^{n}E_{n-1,q}\left(
x\right)  .
\]

\end{theorem}

\begin{theorem}
\label{Thm:e2}The $q$-Euler polynomials $B_{k,q}\left(  x\right)  $ satisfy
the $q$-difference equation%
\begin{gather*}
\frac{1}{2}\dfrac{e_{n-1,q}}{\left[  n-1\right]  _{q}!}D_{q,x}^{n}%
E_{n,q}\left(  x\right)  +\frac{1}{2}\dfrac{qe_{n-2,q}}{\left[  n-2\right]
_{q}!}D_{q,x}^{n-1}E_{n,q}\left(  x\right)  +...+\frac{1}{2}\dfrac
{q^{n-2}e_{1,q}}{\left[  2\right]  _{q}!}D_{q,x}^{2}E_{n,q}\left(  x\right) \\
-\dfrac{1}{2}q^{n-1}D_{q,x}E_{n,q}\left(  x\right)  +xq^{n}D_{q,x}%
E_{n,q}\left(  x\right)  -\left[  n\right]  _{q}E_{n,q}\left(  qx\right)  =0.
\end{gather*}

\end{theorem}

\section{$q$-Genocchi polynomials}

The $q$-Genocchi numbers $g_{n,q}$ can be defined by the generating function%
\[
G_{q}\left(  t\right)  :=\frac{2t}{e_{q}\left(  t\right)  +1}=\sum
_{n=0}^{\infty}g_{n,q}\dfrac{t^{n}}{\left[  n\right]  _{q}!}.
\]
The $q$-Genocchi polynomials $G_{n,q}\left(  x\right)  $ (see \cite{mah1}) can
be defined by the generating function%
\[
G_{q}\left(  x,t\right)  :=\frac{2t}{e_{q}\left(  t\right)  +1}e_{q}\left(
tx\right)  =\sum_{n=0}^{\infty}G_{n,q}\left(  x\right)  \dfrac{t^{n}}{\left[
n\right]  _{q}!},\ \ \ \left\vert t\right\vert <\pi.
\]

\begin{theorem}
\label{Thm:g1}The following linear homogeneous recurrence relation for the
$q$-Genocchi polynomials holds true:%
\[
\frac{1}{2q}\sum_{k=0}^{n-2}\left[
\begin{array}
[c]{c}%
n\\
k
\end{array}
\right]  _{q}g_{n-k,q}q^{k}G_{k,q}\left(  x\right)  +\left[  n\right]
_{q}\left(  xq-\frac{1}{2q}\right)  q^{n-1}G_{n-1,q}\left(  x\right)
+q^{n-1}G_{n,q}\left(  x\right)  -\left[  n\right]  _{q}G_{n,q}\left(
qx\right)  =0.
\]

\end{theorem}

\begin{theorem}
\label{Thm:g2}The $q$-Genocchi polynomials $G_{n,q}\left(  x\right)  $ satisfy
the $q$-difference equation%
\begin{gather*}
\frac{1}{2q}\dfrac{g_{n,q}}{\left[  n\right]  _{q}!}D_{q,x}^{n}G_{n,q}\left(
x\right)  +\dfrac{g_{n-1,q}}{2\left[  n-1\right]  _{q}!}D_{q,x}^{n-1}%
G_{n,q}\left(  x\right)  +...+\dfrac{q^{n-3}g_{2,q}}{2\left[  2\right]  _{q}%
!}D_{q,x}^{2}G_{n,q}\left(  x\right) \\
-\dfrac{q^{n-2}}{2}D_{q,x}G_{n,q}\left(  x\right)  +q^{n-1}G_{n,q}\left(
x\right)  +xq^{n}D_{q,x}G_{n,q}\left(  x\right)  -\left[  n\right]
_{q}G_{n,q}\left(  qx\right)  =0.
\end{gather*}

\end{theorem}

\section{ $q$-Hermite polynomials}

In this section we construct a $q$-Hermite polynomials and give of their some
properties. Also, we derive the three-term recursive relation as well as the
second-order differential equation obeyed by these new polynomials.

We define new $q$-Hermite polynomials $H_{n,q}\left(  x\right)  $ by means of
the generating function%
\begin{align*}
H_{q}\left(  x,t\right)   &  :=H_{q}\left(  t\right)  e_{q}\left(  tx\right)
=\sum_{n=0}^{\infty}H_{n,q}\left(  x\right)  \dfrac{t^{n}}{\left[  n\right]
_{q}!},\\
H_{q}\left(  t\right)   &  :=\sum_{n=0}^{\infty}\left(  -1\right)
^{n}q^{n\left(  n-1\right)  }\dfrac{t^{2n}}{\left[  2n\right]  _{q}%
!!},\ \ \ \ \left[  2n\right]  _{q}!!=\left[  2n\right]  _{q}\left[
2n-2\right]  _{q}...\left[  2\right]  _{q}.
\end{align*}
It is clear that
\begin{align*}
\lim\limits_{q\rightarrow1^{-}}H_{q}\left(  x,t\right)   &  =\lim
\limits_{q\rightarrow1^{-}}H_{q}\left(  t\right)  e_{q}\left(  tx\right)
=e^{tx}\lim\limits_{q\rightarrow1^{-}}\sum_{n=0}^{\infty}\left(  -1\right)
^{n}q^{n\left(  n-1\right)  }\dfrac{t^{2n}}{\left[  2n\right]  _{q}!!}\\
&  =e^{tx}\lim\limits_{q\rightarrow1^{-}}\sum_{n=0}^{\infty}\left(  -1\right)
^{n}\dfrac{t^{2n}}{\left(  2n\right)  \left(  2n-2\right)  ....2}=e^{tx}%
\lim\limits_{q\rightarrow1^{-}}\sum_{n=0}^{\infty}\left(  -1\right)
^{n}\dfrac{t^{2n}}{2^{n}n!}\\
&  =\exp\left(  tx-\frac{t^{2}}{2}\right)  .
\end{align*}
Moreover%
\[
\dfrac{D_{q,t}H_{q}\left(  t\right)  }{H_{q}\left(  qt\right)  }%
=-t\ \ \ \ \ \text{and\ \ \ }D_{q,x}H_{n,q}\left(  x\right)  =\left[
n\right]  _{q}H_{n-1,q}\left(  x\right)  .
\]

\begin{theorem}
\label{Thm:h0}The series form of the $q$-Hermite polynomial is given by%
\[
H_{n,q}\left(  x\right)  =\sum_{k=0}^{\left[  \frac{n}{2}\right]  }%
\frac{\left(  -1\right)  ^{k}q^{k\left(  k-1\right)  }x^{n-2k}}{\left[
2k\right]  _{q}!!\left[  n-2k\right]  _{q}!}%
\]

\end{theorem}

\begin{proof}
Indeed, expanding the generation function $H_{n,q}\left(  x,t\right)  $, we
have%
\begin{align*}
H_{q}\left(  x,t\right)   &  =\sum_{k=0}^{\infty}\left(  -1\right)
^{k}q^{k\left(  k-1\right)  }\dfrac{t^{2k}}{\left[  2k\right]  _{q}!!}%
\sum_{l=0}^{\infty}x^{l}\dfrac{t^{l}}{\left[  l\right]  _{q}!}\\
&  =\sum_{n=0}^{\infty}\sum_{l=0}^{\infty}\frac{\left(  -1\right)
^{k}q^{k\left(  k-1\right)  }x^{l}}{\left[  2k\right]  _{q}!!\left[  l\right]
_{q}!}t^{2k+l}\ \ \ (2k+l=n)\\
&  =\sum_{n=0}^{\infty}\sum_{k=0}^{\left[  \frac{n}{2}\right]  }\frac{\left(
-1\right)  ^{k}q^{k\left(  k-1\right)  }x^{n-2k}}{\left[  2k\right]
_{q}!!\left[  n-2k\right]  _{q}!}t^{n}.
\end{align*}

\end{proof}

\begin{theorem}
\label{Thm:h1}The following linear homogeneous recurrence relation for the
$q$-Hermite polynomials holds true:%
\begin{equation}
H_{n,q}\left(  qx\right)  =xq^{n}H_{n-1,q}\left(  x\right)  -\left[
n-1\right]  _{q}q^{n-2}H_{n-2,q}\left(  x\right)  ,\ \ \ n\geq2. \label{hh1}%
\end{equation}

\end{theorem}

Using the recurrence relation (\ref{hh2}), we get%
\begin{align*}
H_{0,q}\left(  x\right)   &  =1,\ \text{(by definition)}\\
H_{1,q}\left(  x\right)   &  =x,\\
H_{2,q}\left(  x\right)   &  =x^{2}-1,\\
H_{3,q}\left(  x\right)   &  =x^{3}-\left[  3\right]  _{q}x,\\
H_{4,q}\left(  x\right)   &  =x^{4}-\left(  1+q^{2}\right)  \left[  3\right]
_{q}x^{2}+\left[  3\right]  _{q}q^{2}.
\end{align*}

\begin{theorem}
\label{Thm:h2}The $q$-Hermite polynomials $G_{n,q}\left(  x\right)  $ satisfy
the $q$-difference equation%
\begin{equation}
q^{n-2}D_{q,x}^{2}H_{n,q}\left(  x\right)  -xq^{n}D_{q,x}H_{n,q}\left(
x\right)  +\left[  n\right]  _{q}H_{n,q}\left(  qx\right)  =0. \label{hh2}%
\end{equation}

\end{theorem}

In the limit when $q\rightarrow1^{-}$, the equation (\ref{hh2}) is reduced to
the second order differential equation satisfied by the Hermite polynomials.

\end{document}